\documentclass[reqno]{amsart}
\usepackage{float}
\usepackage{amsmath}
\usepackage{graphicx}
\usepackage{latexsym}
\usepackage{amsfonts}
\usepackage{amssymb}
\usepackage{hyperref}
\usepackage{bookmark}

\usepackage{fixltx2e,amsmath}
\MakeRobust{\eqref}
\theoremstyle{plain}
\newtheorem{thm}{Theorem}
\newtheorem{cor}[thm]{Corollary}
\newtheorem{lemma}[thm]{Lemma}
\newtheorem{prop}[thm]{Proposition}
\theoremstyle{definition}

\newtheorem{example}[thm]{Example}

\newtheorem{remark}[thm]{Remark}
\newtheorem{remk}[thm]{Remark}

\begin{document}

\title[Random Walks in Cones : The case of  nonzero drift]{Random Walks in Cones: The case of nonzero drift}

\author[Duraj]{Jetlir Duraj}

\address{Mathematical Institute, University of Munich, Theresienstrasse 39, D--80333
Munich, Germany}
\email{jetlir.duraj@mathematik.uni-muenchen.de}

\begin{abstract}
We consider multidimensional discrete valued random walks with nonzero drift killed when leaving general cones of the euclidian space. We find the asymptotics for the exit time from the cone and study weak convergence of the process conditioned on not leaving the cone.  We get quasistationarity of its limiting distribution. Finally we construct a version of the random walk conditioned to never leave the cone.
\end{abstract}



\maketitle
\section{Introduction, results and discussion}
\subsection{Motivation}
We want to study the multidimensional counterpart of the following one dimensional problem:
\newline \indent
Let $S(n)$ be a real valued random walk with negative drift started at some $x\in (0,+\infty)$. Find the asymptotics of the exit time $\tau_x=\inf\{n\geq 0 : S(n)\leq 0\}$. This is by now a classical result.
For example, in \cite{doney2} the asymptotics is found if the jump of the random walk fulfills the following Cram\'{e}r-type condition :\newline $R(h) = \mathbb{E}[e^{hX}]$ is finite in some $[0, B]$ for $B\leq\infty$, and $0<\lim_{h\uparrow B}\frac{R'(h)}{R(h)}\leq \infty$. The asymptotics is then
\begin{equation}\label{eq:1}
\mathbb{P}(\tau_x>n) \sim V(x)\mu^{-n}n^{-\frac{3}{2}} \text{  as  } n \rightarrow \infty.
\end{equation}
Here, $\mu =\frac{1}{\mathbb{E}[e^{h_0X}]}$ and $h_0$ is the unique solution of $R'(h)=0$ and the assumption on the jump implies in particular that it has a finite second moment. For two real-valed sequences, by the notation $a(n)\sim b(n)$, as $n\rightarrow\infty$, we mean the property $\lim_{n\rightarrow\infty}\frac{a(n)}{b(n)}=1$.
\newline \indent Note that the positive open real halfline is a cone. Our main aim in this paper is to find an analogue to \eqref{eq:1} for random walks in dimensions greater than one, killed when leaving cones in the respective euclidian space and to derive some weak convergence results for the conditioned process. 
\newline  Even though the idea of a Cram\'{e}r condition and use of an exponential change of measure will be helpful for the multidimensional case as well, methodologically the multidimensional problem is different from the one dimensional case. Therefore we first briefly recall the main idea of the study of latter problem, as it is done in \cite{doney2}
\newline Its study is facilitated by the obvious relation $\mathbb{P}(\tau_x>n)=\mathbb{P}(L_n\geq -x)$ with $L_n = \min_{i=1..n} S(i)$. This makes possible the use of the following classical relations (see also the similar equations (2.5) and (2.6) in \cite{doney2})
\begin{equation}\label{eq:2}
\sum_{n\geq 0}s^n \mathbb{E}[e^{sL_n}] = \exp\{\sum_{n\geq 1} s^n a_n(s)\}
\end{equation}
and
\begin{equation}\label{eq:3}
\sum_{n\geq 0}s^n \mathbb{P}(\tau_0 >n) = \exp\{\sum_{n\geq 1} s^n a_n\},
\end{equation}
where $a_n(t) = \frac{1}{n}\{\mathbb{E}[e^{tS(n)}, S(n) < 0]+ \mathbb{P}( S_n \geq 0 )\}$ and $a_n =\frac{1}{n} \mathbb{P}(S_n \geq 0)$. Then getting the asymptotics of $a_n(t)$ and $a_n$ yields the asymptotics of the exit time. To this aim one makes use of a change of measure, for example for $a_n$ we have
\begin{align*}
a_n &= \frac{1}{n}\int_0^{\infty}\mathbb{P}(S_n \in dx) \\
&= \frac{1}{n}\int_0^{\infty}e^{-hx}(\mathbb{E}[e^{hx}])^n\mathbb{P}(\hat{S}_n\in dx) \\ &= \frac{(\mathbb{E}[e^{hx}])^n}{n}\mathbb{E}[e^{-\hat{S}_n}, \hat{S}_n \geq 0].
\end{align*}
Here $\hat{S}_n$ is the driftless random walk gained from an exponential change of measure of $X$ through the density $\frac{e^{hx}}{\mathbb{E}[e^{hx}]}$. It is then easy to see, for example if the random walk is discrete by expanding and using a local limit theorem, that the expectation in the last line is asymptotically $\frac{const}{\sqrt{n}}$.
\newline \indent As already mentioned, there is no hope of some similarly helpful relation as \eqref{eq:2} and \eqref{eq:3} for the multidimensional case. A way of attacking the multidimensional case is supplied by the recent work \cite{dw1}.  In it the authors study the asymptotics of the exit time from a general multidimensional cone for the case of driftless random walks. They use then the asymptotics of the exit time and several sharp probabilistic inequalities to establish local limit theorems for lattice valued driftless random walks, killed when leaving the cone. One could use some of their results after reducing the case of a nonzero drift to that of the zero drift. For this we impose a Cram\'{e}r condition and use an exponential change of measure to turn the nonzero drift random walk into a driftless one already at the beginning. Nevertheless, it turns out that one has to refine and specialize some crucial estimate from the driftless case to accomplish the proof for the nonzero drift case. This refinement is the hardest part of the proof.
\newline \indent The paper is organized as follows. In the next section we introduce the setting, discuss the assumptions made and formally state the results. In the second part of the paper we provide the proofs for our results.

\subsection{Assumptions and Statement of Results}\label{subsec:1.2}

Let $S(n)$ be a $d$-dimensional random walk, $d \geq 2$. Its jumps are i.i.d. copies of some random variable \newline $X = (X_1,X_2,...X_d)$ which takes value in the euclidian lattice $\mathbb{Z}^d$.
\newline
We assume 
\newline
\newline
\textbf{Assumption 1 (Cram\'{e}r condition) :} The set \newline $\Omega 
= \{h \in \mathbb{R}^d : R(h) = \mathbb{E}[e^{h\cdot X}]< \infty\}
 $ has a nonempty interior containing 0 and there exists some nonzero $h$ in the interior such that $\nabla R(h) = \mathbb{E}[Xe^{h\cdot X}] = 0$. 
\newline \newline
\indent We also assume that the random walk is truly $d$-dimensional, i.e. it does not live on a hyperplane of the euclidian space. Thus we impose \newline
\textbf{Assumption 2 (Non Collinearity) :} For every nonzero $c \in \mathbb{R}^d$ we have \newline $\mathbb{P}(c\cdot X = 0)<1$. \newline \newline 
\indent We also assume that the random walk is strongly aperiodic.
\newline \textbf{Assumption 3 (Strong aperiodicity):} $X$ fulfills the following condition:
\newline\newline \indent 
For every $x \in \mathbb{Z}^d$ the smallest subgroup of $\mathbb{Z}^d$ containing 
\begin{equation*}
\{y : y= x+z \text{ with some } z \text{ s.t. } \mathbb{P}(X = z)>0\}
\end{equation*}
is the whole group.
\newline This is just a technical assumption and we will see that it is dispensable for the actual asymptotics of the exit time.\newline

Note that Assumption 1 implies that $c: = R(h)$ is smaller than 1 and that \newline $\mathbb{E}[X] = \nabla R(0)$ is nonzero, since the function $R(h)$ is strictly convex and $C^{\infty}$ in the interior of $\Omega$. Thus we are in the nonzero drift case. Also, due to convexity of $\Omega$ and Assumption 1, it follows by a Taylor expansion of $R(h)$ around $0$ that $\mathbb{E}[X]\cdot h < 0$.
\newline \indent Define $c = \mathbb{E}[e^{h\cdot X}]$ and let $\tilde{X}$ be a random variable with density
\begin{equation}\label{eq:measurechange}
\mathbb{P}(\tilde{X}\in dz) = \frac{1}{c} e^{h\cdot z}\mathbb{P}(X\in dz),
\end{equation}
defined on the same probability space as $S(n)$ is. 
As a consequence of the Cram\'{e}r assumption we know its associated random walk $\tilde{S}(n) = \sum_{i=1}^n \tilde{X}(i)$ is driftless. Clearly, the non-collinearity assumption holds for $\tilde{X}$ again, since we are dealing with an equivalent change of measure. This implies that $\mathbb{E}\left[\tilde{X}\cdot\tilde{X}^t\right]$ is a positive definite matrix. This ensures the existence of an invertible $d\times d$-matrix $M$ such that $\hat{X}=M\tilde{X}$ has $\mathbb{E}\left[\hat{X}\cdot\hat{X}^t\right] = \mathbb{Id} _{d\times d}$, where $\mathbb{Id} _{d\times d}$ is the Identity matrix of dimension $d$. We denote by $\hat{S}(n)$ its corresponding random walk. It has uncorrelated components with zero drift, since also $\mathbb{E}[\hat{X}] = M\mathbb{E}[\tilde{X}]=0$. \newline \indent Due to the Cram\'{e}r condition, we know that $\tilde{X}$ and $\hat{X}$ have all moments. The state space of $\hat{S}(n)$ is $M \mathbb{Z}^d$. It is again strongly aperiodic in its state space. We denote from now on by $\hat{y}$ the vector $M y$ for $y\in \mathbb{Z}^d$. 
For the original random walk we introduce the stopping time
\begin{equation}
\tau_x = \inf\{n\geq 0: x + S(n) \not \in K\}
\end{equation}
and by $\hat{\tau}_{\hat{x}}$ the corresponding stopping time for $\hat{S}(n)$,
\begin{equation}
\hat{\tau}_{\hat{x}} = \inf\{n\geq 0: \hat{x} + \hat{S}(n) \not \in \hat{K}=MK\}.
\end{equation}  
Here $K$ is an open cone containing $h$ of Assumption 1 and $x \in \mathbb{Z}^d$ always. 
\newline
\newline 
\textbf{Assumption 4 (Convexity): } K is convex, that is 
\begin{equation*}
\text{for every } y, z \in K \text{ and } \lambda_1 , \lambda_2 > 0 \text{ the vector } \lambda_1y_1 + \lambda_2y_2 \text{ is in } K
\end{equation*}
This assumption is slightly stronger than the respective assumptions for the driftless case from \cite{dw1} but it is needed in order to be able to prove our results.
\newline \newline \indent
Moreover, our method works only for cones where the series $\sum_{y \in K \cap \mathbb{Z}^d} e^{-h\cdot y}$ is convergent, so that letting $\Sigma := K \cap \mathbb{S}^{d-1}$ we have to impose 

\textbf{Assumption 5:} For every $x \in \partial \Sigma$ we have $|\angle (x,\frac{h}{|h|})| < \frac{\pi}{2}$.\newline
This assumption is not fulfilled for two-dimensional cones of opening bigger or equal to $\pi$, which contain $h$ in its interior. For such cones $h\cdot y \leq 0 $ for some $y \in K\cap \mathbb{Z}^d$ can happen so that $\sum_{y \in K \cap \mathbb{Z}^d} e^{-h\cdot y}$ diverges.  \newline \indent For cones fulfilling Assumption 5, it follows that $h\cdot y$ can be bounded from below by a constant times the norms of $h$ and $y$ for $y \in K$. Thus this assumption can be considered as minimal for our method. This assumption and the fact $\mathbb{E}[X]h<0$ imply that the random walks in consideration have a drift which points outside of the cone.\newline\indent We can not use our method to get the exact asymptotic of the probabilities of a two dimensional random walks conditioned to stay in a half-plane. Nevertheless, under some conditions, we can use \eqref{eq:1} to get the tail asymptotics of the exit time. If the cone has the form $K = \{x\in \mathbb{R}^d|a\cdot x >0\}$ for some nonzero $a\in \mathbb{R}^d$, the jump $X$ has $a\cdot\mathbb{E}[X]<0$, there exist some $B>0$ such that $\mathbb{E}[e^{ah\cdot X}]<\infty$ for $h\in[0,B]$, and finally $\lim_{h\uparrow B}\frac{\mathbb{E}[a\cdot Xe^{ha\cdot X}]}{\mathbb{E}[e^{ha\cdot X}]}>0$, then we are precisely in the conditions of \cite{doney2} for the random walk with jump $a\cdot X$ and can use \eqref{eq:1} to get the asymptotics of the exit time.
\newline \newline As a last restriction, we have to impose some additional regularity on $\partial K$.
\newline \textbf{Assumption 6 (Regularity):} $\partial\Sigma$ is $C^2$ and the solution of 
\begin{equation*}
\left\{\begin{array}{cl} \Delta u = 0, & \mbox{if }x \in MK,\\u(x) =0, & \mbox{if } x \in \partial MK \end{array}\right.
\end{equation*}
is extendable to the respective solution on a bigger cone, which strictly contains $MK$, i.e. there exists some cone $\tilde{K}$ which strictly contains $MK$ and for which $u$ is extendable to a solution of the Dirichlet problem for the cone $\tilde{K}$. It is also clear, $M$ being invertible i.e. a $C^{\infty}$-diffeomorphism, that $\partial (MK\cap S^{d-1})$ is $C^2$ if $\partial\Sigma$ is $C^2$. We note here, that if the original random walk has indipendent components, then $M$ is a diagonal positive definite matrix so that $MK = K$ and the Assumption 6 is then made on the original cone.
\newline \indent This type of assumption is only slightly stronger than the assumptions needed to study the driftless case (see \cite{dw1}) or the continuous counterpart of the driftless case, i.e. Brownian motion killed when leaving the cone (see \cite{banusmits}). $u$ as above is harmonic for killed Brownian motion when leaving the cone $MK$. It is a homogeneous function of some degree $p\geq 1$: that is, there exists $m$ defined on $\hat{\Sigma} = MK \cap \mathbb{S}^{d-1}$ s.t. $u(x) = |x|^p m ( \frac{x}{|x|} )$.
\newline\indent With an eye on applications, our assumptions are not as restrictive as they seem. For example, Assumption 6 is always fulfilled for cones $K$ with real-analytic $\partial \Sigma$ (see references in \cite{banusmits}). In particular, every two dimensional cone of opening of less or equal to $\frac{\pi}{2}$ fulfills all of Assumptions 4, 5 and 6, since $\partial \Sigma$ contains just two points. Moreover, depending on the specific setting of the original problem one wants to study, it may be possible to make a linear transformation of the cone and the jump and reduce the setting to a random walk living in some cone of the form $\tilde{K}\times\mathbb{R}^{d-r}$ (for some $r\leq d$ and $\tilde{K}$ suitable), whose projection on $\mathbb{R}^r$ fulfills our assumptions.
\newline \indent We also note, that for our proof to go through we could weaken Assumption 1 as our proof only needs the existence of a nonzero $h$ with $R(h)<1$, $\mathbb{E}[Xe^{h\cdot X}] = 0$ and $\mathbb{E}[|X|^{\alpha}e^{h\cdot X}]<\infty$, where $\alpha$ is some suitable real number depending on the transformation $M$. Here, $\alpha$ has to be greater than 2 and equal to $p$ if $p$ itself is greater than 2. We have stated Assumption 1 to show the analogy with the one dimensional problem.
\newline \indent In this setting we are able to prove
\begin{thm}\label{thm:thm1} There exist functions $U$ and $U'$ such that as $n\rightarrow \infty$, for $A\subset K\cap\mathbb{Z}^d$
\begin{equation*}\label{eq:thmeq1}
\mathbb{P}(x + S(n) \in A, \tau_x > n) \sim \rho c^nn^{-(p+\frac{d}{2})}U(x)\sum_{y \in A} U'(y).
\end{equation*} 
In particular,
\begin{equation}\label{eq:assneg}
\mathbb{P}(\tau_x > n) \sim \varrho c^nn^{-(p+\frac{d}{2})}U(x).
\end{equation}
\end{thm}
Relation \eqref{eq:assneg} is the multidimensional counterpart of \eqref{eq:1} and $U$, $U'$ are suitable functions defined in the next section. We note that one does not need Assumption 3 for \eqref{eq:assneg}. Indeed, given the original random walk is not strongly aperiodic, we could redefine the probabilities on a suitable subset grid of the euclidian grid and get again \eqref{eq:assneg} through the same calculations. Assumption 3 has been stated for expository reasons.  
The proof of Theorem \ref{thm:thm1} is based on the similar results for the driftless random walk $\hat{S}(n)$, but is not a straightforward application, since the estimates on the tail probability for the exit time in the driftless case are not sharp enough to justify an interchange of limit and sum, which is crucial for the proof of the nonzero drift case. Sharpening this estimate for our setting is the essential in our method. 
\newline \indent Having Theorem \ref{thm:thm1}, we derive from it the asymptotic behavior of the conditioned process. As a first simple Corollary it is immediate to derive
\begin{cor}\label{cor:corollary2}
\begin{equation*}\label{eq:cornjesh}
\frac{\mathbb{P}_x(\tau =n)}{\mathbb{P}_y(\tau =n)} \longrightarrow \frac{U(x)}{U(y)} \quad , \qquad n\rightarrow \infty.
\end{equation*}
\end{cor}
The one dimensional version of this limit is found in \cite{doney}.\newline
In \cite{daley} and \cite{iglehart1} the authors find the limit process for the one-dimensional conditioned random walks to be quasistationary. In our setting, as a simple corollary from Theorem \ref{thm:thm1}, one can check that the respective Yaglom limit exists for the multidimensional case, i.e.
\begin{cor}[Yaglom Limit]\label{cor:yaglom} We have for $x \in K \cap \mathbb{Z}^d$, $A\subset K$ and $n\rightarrow \infty$
\begin{equation*}
\mathbb{P}(x + S(n) \in A|\tau_x >n) \longrightarrow \mu (A)= \kappa \sum_{y \in A\cap \mathbb{Z}^d} U'(y).
\end{equation*}
Here, $\kappa > 0$ is the norming constant so that $\mu$ is a probability distribution on $K$. Furthermore, $\mu$ is quasistationary for the conditioned process.
\end{cor}
Existence of the Yaglom limit is not always given. For example, one sees directly that for driftless random walks $S(n)$ due to Theorem 1 and Theorem 6 in \cite{dw1} we have 
\begin{equation*}
\mathbb{P}(x + S(n) =y|\tau_x >n) \longrightarrow 0, \quad n\rightarrow \infty.
\end{equation*} 
I.e. the Yaglom Limit doesn't exist in this case (it needs to be a genuine probability distribution). For some examples on the fact that even for extensively studied processes like birth-death processes both existence and non-existence can happen, see \cite{vandoorn}. For completeness, we also give in the next section a simple proof of the quasistationarity of $\mu$ by using specific properties of our setting as an alternative to resorting to the general theory of quasistationary distributions. We also show with a simple example that one cannot expect uniqueness of quasistationary distributions in our setting.
\newline \indent In the next section, we prove some simple weak convergence results for the conditioned process. For the one dimensional case, the analysis of such limits in \cite{iglehart1}, \cite{iglehart2} and \cite{daley} is complete. They use the methodology of the one-dimensional problem, which as previously noted, has no bite in the multidimensional setting. Here we prove the following for the exit distribution of the random walk, conditioned on exiting at a specific time: 
\begin{prop}\label{prop:exit}
For $x \in K \cap\mathbb{Z}^d$ and $y \in \mathbb{Z}^d \setminus K$ and $n\rightarrow \infty$
\begin{equation*}\label{eq:dysh}
\mathbb{P}(x + S(\tau_x) = y | \tau_x =n) \longrightarrow \frac{\chi}{1-c} \sum_{z \in K} U'(z) \mathbb{P}(z+S(1)=y),
\end{equation*}
where $\chi$ is a positive constant.
\end{prop} 
For completeness we also prove a result which describes the behavior of bridges for random walks in our setting:
\begin{prop}\label{prop:propdw}
For $A\subset K\cap \mathbb{Z}^d$ finite, $z\in K$ and $t\in (0,1)$
\begin{equation*}
n^{p+\frac{d}{2}}\mathbb{P}(x + S([tn]) \in A|\tau_x >n, x + S(n) = z) \sim\frac{\rho}{(t(1-t))^{p+\frac{d}{2}}}\sum_{y\in A} U(y)U'(y), \quad n\rightarrow \infty.
\end{equation*}
\end{prop}
\indent Finally, Theorem \ref{thm:thm1} makes it possible to construct a Markov chain on $K\cap \mathbb{Z}^d$ conditioned to never leave the cone. Namely, we get the markov chain Z with transition matrix

\begin{equation*}
p(x,y) = \frac{1}{c} \mathbb{P}(x+S(1) = y) \frac{U(x)}{U(y)}, \quad \text{  for }x,y \in K\cap \mathbb{Z}^d
\end{equation*}
by looking at the weak limit of
\begin{equation*}
\mathbb{P}(x+S(1)=y_1,x+S(2)=y_2,\dots,x+S(k)=y_k,\tau_x>n|\tau_x>n)
\end{equation*}
as $n\rightarrow \infty$.
It is then easy to prove the following:
\begin{prop}\label{prop:z}
$Z$ is a strictly stochastic and transient Markov chain on $K\cap \mathbb{Z}^d$.
\end{prop}
$Z$ is probably "the physically right" process to take as a random walk conditioned to never leave the cone and it is important to note that it is not constructed through a \emph{Doob-h}-transform, but involves instead a $c$-harmonic function, for some suitable $c\in(0,1)$. In some of the cases considered in applications, like two-dimensional random walks in two-dimensional cones, there are uncountably many positive harmonic functions (for a proof for the positive quadrant, see \cite{ignarolo}, the extension to a proof for more general cones in two dimensions can be made available by the author upon request). So there are uncountably many possible definitions through \emph{Doob-h}-transforms of a random walk conditioned to stay in a cone in this case. Proposition \ref{prop:z} tells us that none of them is equal to the process gained through the weak limit procedure.
\newline To conclude this section, we comment on a special case which is ubiquitous in applications.
\begin{example}[Two dimensional random walks with negative drift in the positive quadrant] 
There has been extensive research on random walks in the quarter plane. In the monograph \cite{fayolle}, the authors develop analytical and algebraic methods to study such random walks under strong assumptions on the distribution of the jump. These methods have been successfully used, for example in \cite{kurkovaraschel} and \cite{raschel} to study random walks of non-zero drift with small steps. In comparison, our assumptions here are less stringent and we follow a pure probabilistic approach.
\newline \indent In our setting, assume that $d=2$ and that the jump of the random walk has negative drift $m = \mathbb{E}[X] = (m_1,m_2)$ with $m_1,m_2<0$. Also assume that there exists some $h_0 > 0$ such that for 
$\tilde{h}_{i} \leq h_0,\quad i=1,2$ we have
\begin{equation*}
\mathbb{E}[e^{\sum_{i=1}^2 \tilde{h}_i X_i}] <\infty
\end{equation*} 
Assume also the existence of some $h=(h_1,h_2)$ with $h_0>h_i>0,\quad i=1,\dots,d$ such that
\begin{equation*}
\mathbb{E}\left[X_j e^{\sum_{i=1}^2 h_i X_i}\right] = 0 \quad j=1,2.
\end{equation*} 
Last, assume that the random walk is strongly aperiodic and that $X$ has a strictly two-dimensional distribution. Under these conditions, if we take for $K$ the interior of the positive quadrant, all of the Assumptions 1-6 are fulfilled. In particular, Theorem 1 holds and yields the asymptotics for $\mathbb{P}(\tau_x > n)$ for $x \in \mathbb{N}^2$ and all its immediate corollaries apply. In this setting we can give an explicit form for $M$. Namely, using Example 2 from \cite{dw1} one can calculate $M$ this way:
define first
\begin{equation*}
c_i = \mathbb{E}[(\tilde{X}_i)^2].
\end{equation*}
Then 
\begin{equation*}
\alpha = \textbf{Cov} \left( \frac{\tilde{X_1}}{\sqrt{c_1}}, \frac{\tilde{X_2}}{\sqrt{c_2}} \right) \in (-1, +1)
\end{equation*} due to non collinearity. With this, we have then 
\begin{equation*}
M = \frac{1}{\sqrt{1-\alpha^2}}
\begin{pmatrix}
\frac{\cos{\varphi}}{\sqrt{c_1}} & \frac{-\sin{\varphi}}{\sqrt{c_2}} \\
\frac{-\sin{\varphi}}{\sqrt{c_1}} & \frac{\cos{\varphi}}{\sqrt{c_2}}
\end{pmatrix},
\end{equation*}
where $\varphi$ is such that $\sin(2\varphi) = \alpha$. The terms $\frac{1}{\sqrt{c_i}}$ norm the variables $\tilde{X_i}$ into having variance $1$, while the $\cos$ and $\sin$-functions cause a rotation, which makes the components uncorrelated, as can be seen by straightforward computation. In the case of two-dimensional cones the value of $p$ is known and equal to $\frac{\pi}{arccos(-\alpha)}$.
\end{example}

\section{Proofs}
We first introduce some notation which we will use for the rest of the paper. We denote for $n\in \mathbb{N}$
\begin{equation*}
q^{(n)}(x,y) = \mathbb{P}(x+S(n)=y,\tau_x>n) \quad \text{and}\quad
d^{(n)}(x,y) = \mathbb{P}(\hat{x}+\hat{S}(n)=\hat{y},\hat{\tau}_{\hat{x}}>n)
\end{equation*}
and also 
\begin{equation*}
f^{(n)}(x,y) = \mathbb{P}(\hat{x}-\hat{S}(n)=\hat{y},\hat{\tau}_{\hat{x}}>n)
\end{equation*}
We then have for $x,y \in \mathbb{Z}^d$, using the inverse change of measure,
\begin{align*}
q^{(n)}(x,y) &= c^n \mathbb{E}\left[e^{-h\cdot \tilde{S}(n)}\textbf{1}_{\{x + \tilde{S}(n) = y, \tau_x>n\}}\right]\\
= & c^n e^{h\cdot x} \mathbb{E}\left[e^{-h\cdot (\tilde{S}(n)+x)}\textbf{1}_{\{x + \tilde{S}(n) = y, \tau_x>n\}}\right] = c^n e^{h\cdot (x-y)} \mathbb{P}(x + \tilde{S}(n) = y, \tau_x>n).
\end{align*}

This can be expressed in terms of $\hat{S}(n)$ due to invertibility of $M$ as 
\begin{equation}\label{eq:h12}
q^{(n)}(x,y) = c^n e^{h\cdot (x-y)} d^{(n)}(x,y)
\end{equation}
Here we have established a connection between the driftless and the nonzero drift case and can hope to use results from \cite{dw1} to analyse the nonzero drift case. There is one complication for that. 
\newline \indent
To get to the assymptotics of $\mathbb{P}(\tau_x>n)$ we need to consider
\begin{equation*}
\sum_{y \in K\cap\mathbb{Z}^d}  c^n e^{h\cdot (x-y)} d^{(n)}(x,y)
\end{equation*}
for $n\rightarrow \infty$. Since we have the assymptotics of $d^{(n)}(x,y)$ from \cite{dw1}, we would need to interchange the sum and the limit for $n\rightarrow \infty$ to get the desired assymptotics. This cannot be done directly, but can be justified if we sharpen an important auxiliary result from \cite{dw1}. Namely, the authors prove in the driftless case an estimate of the type 
\begin{equation}\label{eq:estimstoprw}
\mathbb{P}(\tau_y > n)\leq C\frac{|y|^p}{n^{\frac{p}{2}}}
\end{equation}
uniformly in $K_{n, \epsilon} = \{x \in K : dist(x,\partial K) \geq n^{\frac{1}{2}-\epsilon}\}$ for some $p>0$ for cones satisfying the requirements of their paper. $K_{n, \epsilon}$ contains "typical" vectors in the cone. Here $u$ is the harmonic function for the Brownian motion killed when leaving the cone and $p$ is the characteristic of the cone $K$, i.e. the degree of homogeneity of $u$ as explained in the first section of the paper. We will now extend this estimate to the whole cone K.

\subsection{A refinement of \eqref{eq:estimstoprw} for the driftless case}
For driftless random walks $S(n), n\geq 0$, which fulfill the moment assumptions of \cite{dw1} we want to prove the following Proposition:
\begin{prop}\label{prop:ndih}
Let K be a $d$-dimensional convex cone with vertex in $0$, which also fulfills Assumption 6 above with $M$ the identity matrix. Then there exists some $C>0$ such that uniformly for every $x \in K$ we have 
\begin{equation*}
\mathbb{P} (\tau_x > n) \leq C \frac{1+|x|^p}{n^{\frac{p}{2}}} 
\end{equation*}
\end{prop}
To prove Proposition \ref{prop:ndih} we make use of two Lemmas.
\begin{lemma}\label{lemma:ndih1}
Let K be a $d$-dimensional convex cone with vertex in $0$, which is also $C^2$. For $\epsilon>0$ sufficiently small there exists some $C > 0$ such that, uniformly for every $y\in K_{n,\epsilon}$,
\begin{equation*}
\mathbb{P} (\tau_y > n) \leq C \frac{u(y)}{n^{\frac{p}{2}}}.
\end{equation*}
\end{lemma}
\begin{proof}
The proof follows ideas from Lemma 20 of \cite{dw1}. Take some $x_0\in K$ fixed such that $|x_0|=1$. Also take some arbitrary but fixed $\delta \in (0,1)$. For any $\epsilon >0$ such that $p\epsilon<\frac{1}{2}$, choose $\gamma \in (p\epsilon, \frac{1}{2})$. Define with this for every $y\in K_{n,\epsilon}$ the vector $y^+ = y + R_0x_0n^{\frac{1}{2}-\gamma}$, where $R_0>0$ is chosen such that $dist(R_0x_0+K,\partial K)>1$. Note that this choice of $R_0$ ensures that $R_0x_0n^{\frac{1}{2}-\gamma}+K\subset K_{n,\gamma}$. Define now the event 
\begin{equation}
A_n = \{\sup_{u\leq n}|S([u])-B(u)|< n^{\frac{1}{2}-\gamma}\}.
\end{equation}
Here, $B$ is a Brownian motion defined on the same probability space as $S(n)$. We know from the coupling method used in Lemma 17 of \cite{dw1} that 
\begin{equation}
\mathbb{P}(A_n^{c})\leq Cn^{-r}
\end{equation}
with $r = r(\delta,\gamma) = \frac{\delta}{2} - 2\gamma - \gamma\delta$.
It follows
\begin{equation}\label{eq:ndih}
\mathbb{P} (\tau_y > n) \leq \mathbb{P} (\{\tau_y > n\}\cap A_n) + O(n^{-r}).
\end{equation}
The result would follow if we could prove the desired inequality for each of the two terms on the right side of inequality \eqref{eq:ndih}. For this, we first note that the respective $u$ of the cone $K$, as defined in subsection \ref{subsec:1.2},  fulfills
\begin{equation}\label{eq:ndih2}
u(y) \geq C(dist(y, \partial K)^p) \text{ ,} \quad u(y) \geq C|y|^{p-1}(dist(y, \partial K)), \quad
\end{equation}
since the cone $K$ is assumed to be convex and $C^2$.
To see that this holds, consult \cite{dw1}, Lemma 19 and \cite{var}, Section 0 and especially Theorem 1 there and its immediate corollaries. Note that $p\geq 1$ for convex cones, according to \cite{var}.
\newline \indent
For $y\in K_{n,\epsilon}$ the first inequality of \eqref{eq:ndih2} implies 
\begin{equation*}
u(y)n^{-\frac{p}{2}} \geq Cn^{-p\epsilon}. 
\end{equation*}
This in turn, implies that $n^{-r} = o(u(y)n^{-\frac{p}{2}})$ if, keeping $\delta$ fixed, $\epsilon$ and then $\gamma$ are chosen small enough. This establishes the desired inequality for the second term on the right hand side of \eqref{eq:ndih}. \newline \indent For the first term, we first note that the choice of $R_0$ implies $\{\tau_y > n\}\cap A_n\subset \{\tau_{y^+}^{bm}>n\}$. Moreover, using a Taylor expansion of $u$, Lemma 7 in \cite{dw1} and the second inequality in \eqref{eq:ndih2} we have for $y\in K_{n,\epsilon}$
\begin{equation*} 
|u(y^{+}) - u(y)| \leq C|y|^{p-1}n^{\frac{1}{2}-\gamma} \leq Cu(y)n^{-(\gamma-\epsilon)}.
\end{equation*}
This implies immediately 
\begin{equation*}
u(y^{+}) = u(y)(1+o(1)) \quad \text{uniformly for} \quad y\in K_{n,\epsilon}.
\end{equation*} 
\indent This last implication and the well-known estimate (see again \cite{banusmits} or \cite{var})
\begin{equation*}
\mathbb{P} (\tau^{bm}_{y} > n) \leq C \frac{u(y)}{n^{\frac{p}{2}}}, \quad\text{ for } y\in K\quad
\end{equation*}
establishes the inequality for the first term on the right side of \eqref{eq:ndih2}.\newline The result follows.
\end{proof}
We also need to prove the following auxiliary result.
\begin{lemma}\label{lemma:ndih2}
There exists a constant $C>0$ such that for $v$ as in Subsection 1.3 of \cite{dw1}, uniformly for every stopping time $T$ adapted to the natural filtration of the random walk, and $n\in \mathbb{N}$
\begin{equation*}
\mathbb{E}[v(x + S(T)), \tau_x \geq T, T \leq n] \leq C(1 + |x|^p) , \quad x\in K .
\end{equation*}
\end{lemma}
\begin{proof}
We will make use of the estimates in Section 2 of \cite{dw1} throughout.
We first prove the simpler estimate 
\begin{equation*}
\mathbb{E}[u(x+S(n)), \tau_x > n] \leq C(1+|x|^p).
\end{equation*}
Indeed, for $v$ as in Subsection 1.3 of \cite{dw1} define the martingale sequence
\begin{equation*}\label{eq:martseq}
\begin{split}
Y_0 &= v(x),\\
Y_{n+1} &= v(x + S(n+1)) - \sum_{k=0}^{n} f(x+S(k)) ,\quad x \in K, \quad n \geq 0.
\end{split}
\end{equation*}
Here $f$ is the mean drift of $v$, i.e.
\begin{equation*}\label{eq:f}
f(x) = \mathbb{E}[v(x+X)] - v(x).
\end{equation*}
From the definition and properties of $Y_k$, for $k \geq 1$ we have 
\begin{align*}
&\mathbb{E}[u(x+S(n)), \tau_x > n] = \mathbb{E}[Y_n, \tau_x>n] + \sum_{k=0}^{n-1} \mathbb{E}[f(x+S(k)), \tau_x>n] \\
&= \mathbb{E}_x[Y_n] - \mathbb{E}[Y_n, \tau_x \leq n] + \sum_{k=0}^{n-1} \mathbb{E}[f(x+S(k)), \tau_x>n] \\
&= u(x) - \mathbb{E}[Y_{\tau_x}, \tau_x \leq n] +  \sum_{k=0}^{n-1} \mathbb{E}[f(x+S(k)), \tau_x>n] \quad \text{(using the martingale property of } Y_n)\\
&= u(x) - \mathbb{E}[v(x + S(\tau_x)) , \tau_x \leq n] + \sum_{k=0}^{n-1} \mathbb{E}[f(x+S(k)), \tau_x>n] +\mathbb{E} [\sum_{k=0}^{\tau_x-1} f(x+S(k)), \tau_x\leq n]\\
& \leq u(x) + \mathbb{E}[|v(x + S(\tau_x))|] + 2\mathbb{E} [\sum_{k=0}^{\tau_x-1} |f(x+S(k))|] \\
&\leq C(1 + |x|^p)
\end{align*}
uniformly in $x$. Here, in the last inequality we have used the properties of $u$ (see Section 1) and (22) and (24) from Lemma 12 in \cite {dw1}. Now set $\sigma_x =\tau_x\wedge n$. Then $\lbrace \tau_x \geq T, T\leq n \rbrace = \lbrace T \leq \sigma_x\rbrace$. Since $Y_k$ is a martingale we get
\begin{equation*}
\mathbb{E}_x[Y_{\sigma_x}]= \mathbb{E}_x[Y_{\sigma_x\wedge T}] = \mathbb{E}_x[Y_T , T \leq \sigma_x] + \mathbb{E}_x[Y_{\sigma_x}, T > \sigma_x].
\end{equation*}
That is 
\begin{equation*}
\mathbb{E}_x[Y_T , T \leq \sigma_x] =\mathbb{E}_x[Y_{\sigma_x}, T \leq \sigma_x].
\end{equation*}
Now the definition of $Y_k$ yields
\begin{align*}
\mathbb{E}[v(x+ S(T)), T \leq \sigma_x] &= \mathbb{E}[Y_T + \sum_{k = 0}^{T-1} f(x + S(k)) , T \leq \sigma_x] \\
& = \mathbb{E}_x [Y_{\sigma_x} , T \leq \sigma_x] + \mathbb{E}[\sum_{k = 0}^{T-1} f(x + S(k)) , T \leq \sigma_x]\\
& = \mathbb{E} [v(x + S({\sigma_x}) , T \leq \sigma_x] - \mathbb{E}[\sum_{k = T}^{\sigma_x-1} f(x + S(k)) , T \leq \sigma_x]\\
& \leq \mathbb{E} [u(x + S({n})) , \tau_x > n] +\mathbb{E} [|v(x + S({\tau_x})|] \\&+ \mathbb{E}[\sum_{k = 0}^{\tau_x-1} |f(x + S(k))| , T \leq \sigma_x].
\end{align*}

By using again (22) and (24) from Lemma 12 in \cite {dw1} and the first part of the proof we get the result.
\end{proof}

\begin{proof}[Proof of Proposition \ref{prop:ndih}]
For any $\epsilon>0$ and $\gamma_{n,x} = \inf \lbrace n \geq 0 | x+ S(n) \in K_{n,\epsilon}\rbrace$, the first entry time in $K_{n,\epsilon}$, we have
\begin{equation*}
\mathbb{P} (\tau_x > n) = \mathbb{P} (\tau_x > n, \gamma_{n,x} \leq n^{1-\epsilon}) + \mathbb{P} (\tau_x > n, \gamma_{n,x} > n^{1-\epsilon}).
\end{equation*}
Using Lemma 14 in \cite{dw1} this can be written as
\begin{equation*}
\mathbb{P} (\tau_x > n) = \mathbb{P} (\tau_x > n, \gamma_{n,x} \leq n^{1-\epsilon}) + O(e^{-Cn^{\epsilon}}).
\end{equation*}
Therefore, it suffices to find an upper bound as suggested, only for the first term. Using Markov property we get 
\begin{equation*}
\mathbb{P} (\tau_x > n, \gamma_{n,x} \leq n^{1-\epsilon}) \leq \int_{K_{n, \epsilon}} \mathbb{P} (x + S(\gamma_{n,x}) \in dy, \tau_x > \gamma_{n,x}, \gamma_{n,x} \leq n^{1-\epsilon})\mathbb{P} (\tau_y > n - n^{1-\epsilon}).
\end{equation*}
With help of Lemma \ref{lemma:ndih1} this yields for $\epsilon>0$ sufficiently small
\begin{equation*}
\mathbb{P} (\tau_x > n , \gamma _{n,x} \leq n^{1-\epsilon} ) \leq  C n^{-\frac{p}{2}} \mathbb{E} [u(x + S(\gamma_{n,x})) , \tau_x > \gamma _{n,x} , \gamma _{n,x} \leq n^{1-\epsilon}].
\end{equation*}
Now the result follows immediately from Lemma \ref{lemma:ndih2}.
\end{proof}

\subsection{Proof of Theorem \ref{thm:thm1}}

Take some $A\subset K\cap \mathbb{Z}^d$ and a real number $R>0$. Then we have 
\begin{align*}
\frac{\mathbb{P}(x+S(n) \in A, \tau_x >n)}{\rho c^n n^{-p -\frac{d}{2}}e^{h\cdot x} V(\hat{x})
 \sum_{y \in A}  e^{-h\cdot y} V^{'}(\hat{y})}&\\
=\frac{\sum_{y \in \mathbb{N}^2 , |y| \leq R} n^{p+\frac{d}{2}} e^{h\cdot (x-y)} d^{(n)}(x,y) }{\rho e^{h\cdot x} V(\hat{x})
 \sum_{y \in A}  e^{-h\cdot y } V^{'}(\hat{y})} &+
\frac{\sum_{y \in A , |y| > R} n^{p+\frac{d}{2}} e^{ h\cdot (x-y) } d^{(n)}(x,y)}{\rho e^{h\cdot x} V(\hat{x})\sum_{y \in A}  e^{- h\cdot y} V^{'}(\hat{y})}
\end{align*}
Taking $\lim_{R\rightarrow \infty}\lim_{n\rightarrow \infty}$ in the first summand gives 1 due to \begin{equation*}\label{eq:assympdiscrete}
\mathbb{P}(\hat{x} + \hat{S}(n) = \hat{y}, \tau_{\hat{x}} >n) \sim \rho \frac{V(\hat{x})V'(\hat{y})}{n^{p+\frac{d}{2}}} \quad, \quad n \rightarrow \infty.  
\end{equation*} 
This result is contained in Theorem 6 in \cite{dw1}. Here $V$ is the harmonic function constructed in \cite{dw1} for the driftless random walk $\hat{S}(n)$, killed when leaving the cone $MK$ and $V'$ the respective harmonic function for $-\hat{S}(n)$. They both satisfy an estimate of the type $|V(x)|\leq C(1+|x|^p)$ (see Lemma 13 in \cite{dw1}). We now prove that 
\begin{equation}\label{eq:trtr}
\lim_{R\rightarrow \infty}\lim_{n\rightarrow \infty} \sum_{y \in A , |y| > R} n^{p+\frac{d}{2}} e^{ -h\cdot y } d^{(n)}(x,y) = 0.
\end{equation} 
This would finish the proof of Theorem \ref{thm:thm1}. Setting $m=[\frac{n}{2}]$ we have with Markov property
\begin{equation*}
\begin{split}
d^{(n)}(x,y) &= \sum_{z\in\mathbb{N}^2} d^{(m)}(x,z)d^{(n-m)}(z,y)\\
&= \sum_{z\in\mathbb{N}^2} d^{(m)}(x,z)f^{(n-m)}(y,z),
\end{split}
\end{equation*}
where in the second step, we have used time reversion\footnote{see also the similar calculation in Section 6 of \cite{dw1}}. Namely, we have 
\begin{align*}
&\mathbb{P}(\hat{x} + \hat{S}(n) =\hat{y}, \hat{\tau}_{\hat{x}} >n) = \sum_{z \in \mathbb{N}^2}\mathbb{P}(\hat{x} + \hat{S}(m) =\hat{z}, \hat{\tau}_{\hat{x}} >n)\mathbb{P}(\hat{z} + \hat{S}(n-m) =\hat{y}, \hat{\tau}_{\hat{x}} >n-m)\\
&=\sum_{z \in \mathbb{N}^2}\mathbb{P}(\hat{x} + \hat{S}(m) =\hat{z}, \hat{\tau}_{\hat{x}} >n)\mathbb{P}(\hat{y} - \hat{S}(n-m) =\hat{z}, \hat{\tau}_{\hat{x}} >n-m).
\end{align*} Now use Lemma 27 of Section 6 in \cite{dw1} on $d^{(m)}(x,z)$ to get
\begin{equation*}
d^{(n)}(x,y)\leq n^{-(\frac{p}{2}+\frac{d}{2})}C(x)\mathbb{P}(\hat{\tau}_{\hat{y}}'>n-m)
\end{equation*}
with $\hat{\tau}'_{y} $ the respective exit time for $-\hat{S}(n)$. Using Proposition \ref{prop:ndih} and recalling that $\hat{y}=My$ we get for a suitable $C'>0$
\begin{equation*}
d^{(n)}(x,y)\leq C(x) \frac{C'(1+|y|^p)}{n^{p+\frac{d}{2}}}.
\end{equation*}
If we take $U(x)=e^{h\cdot x}V(\hat{x})$ and $U'(x)=e^{-h\cdot x}V'(\hat{x})$ this establishes \eqref{eq:trtr} and the proof of Theorem \ref{thm:thm1} finished. 
\begin{remark}
As one can see from the proof, the fact that both the probabilities of the killed random walk and the tail probability of the exit time from the cone have the same algebraic order in the  asymptotics is due to the limit in \eqref{eq:trtr} being zero. For cones not fulfilling Assumption 5, we expect \eqref{eq:trtr} to be nonzero, i.e. the algebraic orders of the asymptotics of the probabilities and the exit time to be different. An example could be constructed taking random walks in dimension $d\geq 2$, which fulfill the conditions mentioned in the discussion after Assumption 5, killed when leaving suitable subspaces of $\mathbb{R}^d$. For finite $A$, Theorem \ref{thm:thm1} yields then the algebraic order $c^nn^{-(p+\frac{d}{2})}$ for the asymptotics of the probabilities, but the algebraic order for the tail asymptotics of the exit time is $c^nn^{-\frac{3}{2}}$, as readily follows by the discussion after Assumption 5.
\end{remark}
We can proceed now with corollaries from this result.

\subsection{Weak convergence results from Theorem \ref{thm:thm1}}

\begin{proof}[Proof of Corollary \ref{cor:corollary2}]
It follows easily from \eqref{eq:assneg} that 
\begin{equation}\label{eq:corn}
\frac{\mathbb{P}_x(\tau_x =n)}{\mathbb{P}_x(\tau_x > n)}  \longrightarrow \frac{1-c}{c} \quad , \qquad n\rightarrow \infty.
\end{equation}
The result is now immediate.
\end{proof}
Another trivial remark following from \eqref{eq:assneg} is 
\begin{remk}\label{remk:njesh}
$\mathbb{E}[\exp(\delta \tau_x)] < \infty $ for $\delta \leq -\ln{c}$ and $+\infty$ otherwise.
\end{remk}
This is also known from other results about quasistationarity, see for example \cite{tweedie} or \cite{collet}.
\newline Corollary \ref{cor:corollary2} helps us establish the proof of Proposition \ref{prop:exit}:
\begin{proof}[Proof of Proposition \ref{prop:exit}]
The proof follows closely the one of Theorem \ref{thm:thm1}. We have with Markov property
\begin{equation*}
\mathbb{P}(x + S(\tau_x) = y , \tau_x =n) = \sum_{z\in K \cap \mathbb{Z}^d} q^{(n-1)}(x,z)\mathbb{P}(z+S(1)=y).
\end{equation*}
It follows, that
\begin{equation*}
\begin{split}
&\frac{\mathbb{P}(x + S(\tau_x) = y , \tau_x =n)}{\rho c^{n-1} (n-1)^{-p-\frac{d}{2}}e^{h\cdot x}V(\hat{x})\sum_{z \in  K\cap\mathbb{Z}^d} e^{-h\cdot z} V'(\hat{z}) \mathbb{P}(z+S(1)=y)} = \\ &
\frac{\sum_{z\in \mathbb{N}^2, |z|\leq R}\frac{(n-1)^{p+\frac{d}{2}}}{c^{n-1}}q^{(n-1)}(x,z)\mathbb{P}(z+S(1)=y)}{V(\hat{x})e^{h\cdot x}\sum_{z\in K\cap{\mathbb{Z}^d}} e^{-h\cdot z} V'(\hat{z}) \mathbb{P}(z+S(1)=y)} \quad+
\\&\frac{\sum_{z\in K\cap{\mathbb{Z}^d}, |z|>R}\frac{(n-1)^{p+\frac{d}{2}}}{c^{n-1}}q^{(n-1)}(x,z)\mathbb{P}(z+S(1)=y)}{V(\hat{x})e^{h\cdot x}\sum_{z\in K\cap{\mathbb{Z}^d}} e^{-h\cdot z} V'(\hat{z}) \mathbb{P}(z+S(1)=y)}
\end{split}
\end{equation*}
The first ratio goes to 1 under $\lim_{R\rightarrow \infty}\lim_{n\rightarrow \infty}$ and the second goes to zero by similar reasoning as in the proof of Theorem \ref{thm:thm1}. We thus have
\begin{equation*}
\mathbb{P}(x + S(\tau_x) = y , \tau_x =n) \sim \rho c^{n-1} (n-1)^{-p-\frac{d}{2}}e^{h\cdot x}V(\hat{x})\sum_{z\in K\cap{\mathbb{Z}^d}} e^{-h\cdot z} V'(\hat{z}) \mathbb{P}(z+S(1)=y).
\end{equation*}
Using \eqref{eq:corn} and \eqref{eq:assneg} we get the result.
\end{proof}
Again Theorem 1 and its proof yield an easy proof of Proposition \ref{prop:propdw}, which is similar to a result contained in Theorem 6 of \cite{dw1}
\begin{proof}[Proof of Proposition \ref{prop:propdw}] 
Markov property implies
\begin{align*}
&n^{p+\frac{d}{2}}\mathbb{P}(x + S([tn]) \in A|\tau_x >n, x + S(n) = z) \\&= n^{p+\frac{d}{2}} \sum_{y\in A} \frac{q^{([tn])}(x,y)q^{(n-[tn])}(y,z)}{q^{(n)}(x,z)}.
\end{align*} 
Due to finiteness of $A$ we can use Theorem 6 in \cite{dw1} and get the result.
\end{proof}
\subsection{Comments on quasistationarity}
Quasistationarity can be understood as the long-term behavior of the process conditioned on survival, when it is known that the process will be killed a.s. in the distant future. It has many applications in more practical sciences in modeling phenomena, which will one day come to an end, but so late in time that the behavior until extinction is important to study.
\newline \indent It is well-known that for the one-dimensional problem stated at the beginning of the paper, the Yaglom limit exists. 
Also well-known from the quasistationarity literature (see \cite{collet} or \cite{meleard} for a proof) is the fact that its existence implies the existence of a quasistationary distribution (QSD) for the killed process. I.e. in our case $\mu$ is quasistationary for the random walk killed when leaving the cone. This trivially establishes the second statement in Corollary \ref{cor:yaglom}.\newline\indent
Here for completeness we give another proof of this fact by using the properties of our special setting, namely the $harmonicity$ of $V'$ for its respective driftless killed random walk. Indeed, this property combined with a time inversion implies the following:
\begin{align*}
\mathbb{P}_{\mu}(S(1) =y, \tau > 1) &= \sum_{x \in K\cap\mathbb{Z}^d} \mu(x)\mathbb{P}_x(S(1) =y) = \sum_{x \in K\cap\mathbb{Z}^d}  \kappa e^{-h\cdot x} V^{'}(\hat{x}))\mathbb{P}(x+S(1) =y)\\
& = c\kappa e^{-h\cdot y}\sum_{x \in K\cap\mathbb{Z}^d}  V^{'}(\hat{x})\mathbb{P}(\hat{y} - \hat{S}(1) =\hat{x})\\
&= c\kappa e^{-h\cdot y }  V^{'}(\hat{y}) \\
&= c \mu(y).
\end{align*}
Summation over $y\in K$ yields
\begin{equation*}
\mathbb{P}_{\mu}(\tau > 1) = c,
\end{equation*}
and from this quasistationarity immediately follows.
\newline\indent
An interesting question is whether there is a unique QSD and if not, how many there are. In the one dimensional case, it is well-known that for some typical one dimensional random walks in $\mathbb{N}$ with negative drift there are uncountably many QSD-s (see \cite{ferrari} for more details). This implies the following: if we have a two dimensional random walk with independent components, such that each component has a QSD, then we automatically have a QSD for the two dimensional walk. Indeed, let for example $S_1(n)$ and $S_2(n)$ be two one dimensional random walks in $\mathbb{Z}$ with negative drift, killed when leaving $\mathbb{N}$. Let these have respectively $\mu_1$ and $\mu_2$ as QSD. Then $\mu = \mu_1 \times \mu_2$ is a QSD for the two dimensional random walk $\tilde{S}(n) = (\tilde{S_1}(n),\tilde{S_2}(n))$, which has indipendent components and marginal distributions as those of $S_1$, resp. $S_2$, killed when it leaves $\mathbb{N}^2$. We omit the easy calculation needed to show this. Moreover, a condition of fast return from infinity as required in \cite{martinmartinvillemonais}, which is sufficient in the one-dimensional case for establishing uniqueness of the quasistationary distribution, is not fulfilled in our setting. Therefore we can conjecture that typical random walks on $\mathbb{Z}^d$ with nonzero drift, killed when leaving a cone $K$ have uncountably many QSD-s.
\newline \indent We finish by defining a variant of the \emph{process conditioned to never leave the cone}. See subsection \ref{subsec:1.2} for the idea of the construction.
\subsection{A random walk conditioned to never leave the cone}
\begin{proof}[Proof of Proposition \ref{prop:z}]
We can use the Markov property and \eqref{eq:assneg} to get for $n\rightarrow \infty$
\begin{equation*}
\begin{split}
&\mathbb{P}(x+S(1)=y_1,x+S(2)=y_2,\dots,x+S(k)=y_k,\tau_x>n|\tau_x>n) \\ &\longrightarrow \frac{1}{c^k}q(y_1,y_2)\dots q(y_{k-1},y_{k})e^{h\cdot (y_k-x)}\frac{V(\hat{y}_k)}{V(\hat{x})}.
\end{split}
\end{equation*}
From this, we see that the stochasticity of the Markov chain is equivalent to 
\begin{equation}
ce^{h\cdot x} V(\hat{x}) = \sum_{y\in K \cap\mathbb{Z}^d} q(x,y)e^{h\cdot y}V(\hat{y}).
\end{equation}
This is just $c$-harmonicity of $U$.
Inserting $q(x,y)$ we get after some cancelling
\begin{equation*}
1= \sum_{y\in K \cap\mathbb{Z}^d} p(x,y) \Longleftrightarrow \sum_{y\in K \cap\mathbb{Z}^d} d(x,y)V(\hat{y}) = V(\hat{x})
\end{equation*}
This is the harmonicity of $V$ w.r.t. $\hat{S}(n)$, killed when it leaves the positive quadrant. Therefore, stochasticity follows. The $n$-th power of the transition matrix of $Z$ is
\begin{equation}
p^{(n)}(x,y) = \frac{1}{c^n}q^{(n)}(x,y)e^{h\cdot(y-x)}\frac{V(\hat{y})}{V(\hat{x})}.
\end{equation}
Therefore, using again Theorem 6 in \cite{dw1}, we see that
\begin{equation}
p^{(n)}(x,y) \sim \rho \frac{V(\hat{y})V'(\hat{y})}{n^{p+1}}.
\end{equation} 
Since $p\geq 1$, the Green function of the Markov chain is always finite. Transience follows. 
\end{proof}

\textbf{Acknowledgment}.  I am thankful to Dr. Vitali Wachtel and to an anonymous referee for their many useful comments.

\end{document}